\documentclass[12pt]{article}

\usepackage{amsmath}        
\usepackage{amsfonts}       
\usepackage{amssymb}        
\usepackage{stmaryrd} 
\usepackage{hyperref}

\usepackage[a4paper, total={6in, 8in}]{geometry}

\newcommand\F{\mbox{I\kern-2pt F}}
\newcommand\cA{{\cal A}}

\newcommand\cF{{\cal F}}
\newcommand\cG{{\cal G}}
\newcommand\cH{{\cal H}}

\newcommand\cL{{\cal L}}
\newcommand\cB{{\cal B}}

\newcommand\cX{{\cal X}}
\newcommand\cD{{\cal D}}

\newcommand\cP{{\cal P}}

\def\bbr{{\mathbb R}}
\def\bbn{{\mathbb N}}

\def\bbf{{\mathbb F}}
\def\bbg{{\mathbb G}}

\def\R{{\mathbb R}}

\newcommand\E{{\bf E}}
\newcommand\1{{\bf 1}}

\newtheorem{theo}{Theorem}[section]

\newtheorem{lemm}[theo]{Lemma}

\newcommand\beq{\begin{equation}}
\newcommand\eeq{\end{equation}}
\newcommand\bea{\begin{eqnarray}}
\newcommand\eea{\end{eqnarray}}
\newcommand\bean{\begin{eqnarray*}}
	\newcommand\eean{\end{eqnarray*}}

\newenvironment{proof}{{\it Proof.}}{{$\square$}}

\newcommand{\wh}{\widehat}

\newcounter{assumptioncounter} 
\renewcommand{\theassumptioncounter}{A.\arabic{assumptioncounter}}
\newcommand{\assumptionnumber}{}



\title{Skorokhod Transition in the Conic Market Model}
\author{Artur Sidorenko}
\date{\today}

\begin{document}

\maketitle

\begin{abstract}
This paper examines the applicability of the Skorokhod representation theorem in filtrated probability spaces for the utility maximization problem in the Kabanov conic model of multi-asset markets with proportional transaction costs. A key challenge is that the theorem does not necessarily preserve adaptedness, meaning that solutions obtained on an auxiliary probability space may not correspond to those on the original one. We establish that, under mild conditions, the Bellman functionals remain consistent across different probability spaces.
\end{abstract}

{\bf Keywords} {Markets with transaction cost $\cdot$ Portfolio optimization $\cdot$ Skorokhod representation $\cdot$ Weak convergence}

\section{Introduction}

This paper explores a delicate issue concerning the applicability of the Skorokhod representation theorem in filtrated probability spaces, particularly in the context of mathematical finance. A well-known challenge is that the Skorokhod representation does not necessarily preserve the adaptedness of stochastic processes. Consequently, the solution to a problem obtained on a new probability space may bear little resemblance to the original problem on the initial probability space.

We investigate the applicability of Skorokhod transfer in the context of multi-asset utility maximization under proportional transaction costs. To establish the validity of Skorokhod’s theorem in filtrated settings, two primary approaches have been proposed in the literature. One method requires that the underlying filtration is generated by a process with independent increments \cite{Chau-Rasonyi2017, CR2020}. The second approach relies on the prediction process (see, e.g., \cite{Aldous1981, Jakubowski-Slominski1986}) and the concept of extended weak convergence \cite{BDD} for the price process. Additionally, we note the existence of a version of the Skorokhod representation theorem for filtrated probability spaces \cite{Hoover1991}, which is based on the notion of convergence in adapted distribution \cite{Hoover-Keisler1984}, a stronger form of extended weak convergence.

While the paper \cite{BDD} employs the Skorokhod transfer in filtrated spaces, its treatment remains implicit and lacks  explicit formalization. This gap highlights the need for a more rigorous examination, which we address in this study.

We examine both of the aforementioned approaches by considering two models of financial markets, where the stochastic bases and price processes define the respective filtrated probability spaces. We establish that, under mild conditions, the portfolio optimization problems formulated on different probability spaces lead to identical Bellman functionals.

To streamline the theoretical investigation, we employ the geometric framework of Y. Kabanov \cite{K, KL2002, KabS}, which provides an efficient approach to modeling multi-asset financial markets with transaction costs.

The paper is organized as follows. Section \ref{section:setting} introduces the setting of the utility maximization problem. Section \ref{section:main} presents the main results. Auxiliary statements are relegated to the Appendix.

\section{Setting}
\label{section:setting}

We use the notation $\bbr_{++} := (0, \infty)$. Let $p, d \in \bbn$ and $T \in \bbr_{++}$. We denote $\cD^d_T$ the space of all RCLL (right-continuous with left limits) functions $f: [0, T] \mapsto \bbr^d$.
We rely on the notation of Section 3.6 of \cite{KabS}. 

We consider a model ${\bf M} := (\Omega, \cF, \bbf, P, Y, S)$ of a financial market, where the quadruple $(\Omega, \cF, \bbf, P)$ is a stochastic basis, filtration $\bbf := (\cF_t)_{t=0}^T$ is generated by the null sets of $\cF$ and an $\bbr^p$-valued process $Y$ with RCLL paths, $S$ is an $\bbr^d_{++}$-valued adapted process with continuous paths, $S_0 := {\bf 1}$ and $S^1 = 1$. The latter process is interpreted as prices of $d$ tradable assets, and the first asset is called a numeraire. We fix a closed proper convex $K \subset \bbr^d$ and $\bbr^d_{++} \subset K$. Recall that a closed convex cone $K$ is proper if $K \cap (-K) = \{0\}$. The dual cone $K^* = \{ y \in \bbr^d \colon xy \geq 0 \, \forall x \in K  \}$. It is easily seen that $K^* \setminus \{0\} \subset \bbr^d_{++}$. Thanks to the bipolar theorem, the second dual $K^{**} = K$. 

A function $f: [0, T] \to \bbr^d$ is called $K$-increasing if $f(t) - f(s) \in K$. A function $f: [0, T] \to \bbr^d$ is called $K$-decreasing if $f$ is $-K$-increasing. We put $\cH_K$ the set of all $K$-decreasing paths from $\cD^d_T$. Any adapted  $B = (B_t)_{t = 0}^T$ process with paths from $\cH_K$ is called a strategy. Since $K$ is proper, the paths of $B$ are of finite variation, see Lemma \ref{lemm:var} in Appendix. In this setting, the $j$-th entry of the difference $B_t - B_s$ is the amount of asset $j$ bought/sold in the units of the numeraire throughout the interval $(s, t]$, and $B_0$ is the initial trade at time $0$.

Now we formulate the portfolio optimization problem. Let $x \in K$ be the initial value. We see the portfolio wealth as a vecrtor-valued process, where each entry is corresponded to the amount of asset $i$. This amount can be defined in both physical units and units of the numeraire. The portfolio wealth in physical units is defined as
$$
\wh V^j_t := x^j + \int_{[0, t]} \frac{1}{S^j_s} dB^j_s,
$$
where the integration is understood in the Riemann--Stieltjes sense. By convention, $B_{0-} = 0$ and $B_0$ is the Dirac measure at the initial point 0. In a more concise notation, the latter equality can be written as
$$
\wh V := x + \frac{1}{S} \cdot B,
$$
where $\cdot$ is the component-wise Riemann-Stieltjes integration. The portfolio wealth in the numeraire units is
$$
V^j = S^j \wh V^j.
$$
or
$$
V = S \odot \wh V,
$$
where $\odot$ is the Hadamard multiplication. If $S$ is a semi-martingale, then
$$
V = x + V \cdot Z + B,
$$
where $Z^j = 1 + 1/S^j \cdot S^j$ is the stochastic logarithm of $S^j$ and the stochastic integration of vector-valued processes is understood in the component-wise sense.

To undeline the dependence of $V$ on $x$ and $B$, we use the notation $V^{x, B}$.

A strategy $B$ is admissible if $V^{x, B}_t \in K$ a.s. for every $t \in [0, T]$. Since the paths of $V$ are RCLL, this is equivalent to $V^{x,B} \in K$ up to evanescence. We denote $\cA(x)$ the class of all admissible strategies. 

Let $U: K \mapsto \R \cup \{- \infty\}$ be a mapping interpreted as the utility function. We consider the following portfolio optimization problem
$$
u(x) := \sup_{B \in \cA(x)} \E \, U(V^{x, B}_T).
$$
We call $u(x)$ the Bellman function. To underline the model ${\bf M}$, we use the notation $u(x, {\bf M})$.

\section{Main result}
\label{section:main}

The main problem of this paper is as follows. Let ${\bf M}$ and ${\bf \tilde M}$ be two different models with the probability distributions $\cL_P(S) = \cL_{\tilde P}(\tilde S)$. We are interested in sufficient conditions for the Bellman functions of different models to coincide, i.e. conditions when $u(x, {\bf M}) = u(x, {\bf \tilde M})$. 


We impose the following assumption on model ${\bf M}$:

\smallskip
{\bf H.} The process $Y$ has independent increments (by convention, $Y_{0-} = 0$ and $Y_0$ is considered as an increment).

\smallskip

Without loss of generality, one may assume that $\Omega$ carries a random variable $\xi$ independent of $\cF^{Y}_T$ with uniform distribution in $[0, 1]$.  We define $\tilde Y := (Y, \xi {\bf 1}_{\llbracket 0, \infty \llbracket})$. Aside from the model ${\bf M}$, we consider a model ${\bf  N} = (\Omega, \cF,  \bbg := (\hat \cG_t)_{t=0}^T, P, Y, S)$, where $\cG_t = \sigma \{ \cF_t, \xi \}$.
Note that if ${\bf M}$ satisfies {\bf H}, then so does ${\bf  N}$.


We denote 
\beq
Y^{t} := Y \1_{\llbracket 0, t \llbracket} + Y_t \1_{\llbracket t, \infty \llbracket}  
\eeq
and
\beq
_t \! Y := (Y - Y_t) \1_{\llbracket t, \infty \llbracket}.  
\eeq

While proving the next two theorems, we adopt the techniques of \cite{Chau-Rasonyi2017}.

\begin{theo}
\label{theo:rand}
Let model ${\bf M}$ satisfy {\bf H}. Then for every $x \in {\rm int} \, K$
$$
u(x, {\bf M}) = u(x, {\bf N}).  
$$
\end{theo}

\begin{proof}
    Let $B \in \cA(x, {\bf  N})$. Thanks to the measurability, for some Borel function $f: \cD^m_T \times [0, 1] \mapsto \cD^d_T$, the process $B = f(Y, \xi)$. Our goal is to verify that for a.a. $a \in [0, 1]$, the process $\bar B(a) := f(Y, a) \in \cA(x, {\bf M})$. 
    
    Fix $t \in [0, T]$. It is easily seen that $B_t = \pi_t(f(Y, \xi))$ where $\pi_t(x) = x_t$. On the other hand, $B_t$ is $\tilde \cF_t$-measurable implying $\pi_t(f(Y, \xi)) = f_t(Y^t, \xi)$ for some Borel function $f_t$. By independence of $Y$ and $\xi$ and the Fubini theorem, $f_t(Y^t, a) = \pi_t(f(Y, a))$ for a.a. $a \in [0, 1]$. It follows that for a dense countable set $I \subset [0, T]$ with $\{0, T \} \in I$, r.v. $\bar B_t(a)$ are $\cF_t$-measurable for a.a. $a \in [0, 1]$. Since $B_t = \lim_{I \ni s \to t+} B_s$ for $t \neq T$ and ${\bf F}$ is right-continuous, $\bar B(a)$ is adapted for a.a. $a \in [0, 1]$. On the other hand, 
    $$
    \1_{x + 1/S \cdot B \in \wh K_t, \, \forall t \in [0, T] } = g(Y, \xi)
    $$
    for some Borel function $g$. Immediately, $g(Y, \xi) = 1$ a.s. Again, due to the Fubini theorem, there exists a set $\Omega' \subseteq \Omega$ of probability one such that $g(Y(\omega), a) = 1$ for a.a. $a \in [0, 1]$ for all $\omega \in \Omega'$. By the same technique, we verify that $\bar B(a)$ is $K$-decreasing for a.a. $a \in [0, 1]$ for all $\omega \in \Omega''$ from a set $\Omega''$ of probability one. Thus, we have established that $\bar B(a) \in \cA(x, {\bf M})$ for a.a. $a \in [0, 1]$.

    Finally, fix $\varepsilon > 0$. Then, for some strategy $B \in \cA(x, {\bf N})$,
    \begin{multline*}
    u(x,  {\bf N}) - \varepsilon \leq \E U(V^{x, B}_T, S) = \E \left[ \E \left[ U(V^{x, B}_T, S) \, | \, \xi \right] \right] = \\ 
     \int_{[0, 1]} \E  U(V^{x, \bar B(a)}_T, S)  \, da \leq \int_{[0, 1]} u(x, {\bf M}) da = u(x, {\bf M}).
    \end{multline*}
    As $\varepsilon$ is arbitrary, the assertion follows.
\end{proof}

Note that for more complicated goal functionals, the Bellman functionals may change after the initial enlargement of the filtration, see \cite{Carassus-Rasonyi2015} for a detailed treatment.

The main result of this work is 
\begin{theo}
\label{theo2}
    Let ${\bf M}$ and ${\bf \tilde M}$ be two different models that satisfy {\bf H} and the laws $\cL_P(Y, S) = \cL_{\tilde P}(\tilde Y, \tilde S)$. Then $u(x, {\bf M}) = u(x, {\bf \tilde M})$.
\end{theo}

\begin{proof}
    Let $B \in \cA(x, {\bf M})$.  Due to Lemma 31 of \cite{Chau-Rasonyi2017}, the strategy can be represented as a Borel function $f: \cD^k_T \times [0,1] \to \cD^d_T$ of $Y$ and the r.v. $\xi$ as follows: $B = f(Y, \xi)$. Put $\tilde B := f(\tilde Y, \tilde \xi)$. Our aim is to verify that $\tilde B \in \cA(x,  {\bf \tilde  N})$, where ${\bf \tilde N}$ is the model with additional randomization. $\tilde B$ is $K$-decreasing since $\cL_{P}(B)$ charges probability one to $\cH_K$. Analogously, $\cL_{P} (S, B)$ charges probability one to a set 
    $$
    \left\{ (x, b) : \, x \in C^d_{++}([0, T]) \times \cH_K, \, x_t \odot (1/x \cdot b_t) \in K \; \forall t \in [0, T] \right\},
    $$
    implying that $B'$ satisfies the admissibility property. 
    
    It remains to establish that $\tilde B_t$ is $\tilde \cG_t$-measurable. Note that
    $$
    \cL(B'_t, Y'_u - Y'_t) = \cL(B'_t) \otimes \cL(Y'_u - Y'_t)
    $$
    for any pair $0 \leq t \leq u \leq T$. Following \cite{Chau-Rasonyi2017}, we put 
    $$
    \tilde Y^{t} := \tilde Y \1_{\llbracket 0, t \llbracket} + \tilde Y_t \1_{\llbracket t, \infty \llbracket}  
    $$
    and
    $$
    _t \! \tilde Y := (\tilde Y - \tilde Y_t) \1_{\llbracket t, \infty \llbracket}.  
    $$
    Then $\tilde B_t$ is $\sigma\{ \tilde Y^{t}, \,_t \! \tilde Y, \tilde \xi \}$-measurable and independent of $\,_t \! \tilde Y$. Besides, $\,_t \! \tilde Y$ is independent of $(\tilde Y^{t}, \, \tilde \xi)$. Due to Lemma 29 of \cite{Chau-Rasonyi2017}, $\tilde B_t$ is $\sigma\{ \tilde Y^{t}, \tilde \xi \}$-measurable. 
    As the choice of $B$ is arbitrary,
    $$
    u(x, {\bf M}) \leq u(x, {\bf \tilde N}).
    $$
    By virtue of Theorem \ref{theo:rand}, 
    $$
    u(x, {\bf M}) \leq u(x, {\bf \tilde M}).
    $$
    By swapping the positions of ${\bf M}$ and ${\bf \tilde M}$, we arrive at the required assertion.
\end{proof}

The above results admits a certain generalization if the underlying process $Y$ does not have independent increments. Recall the definition of the prediction process (see, e.g. \cite{Aldous1981, Jakubowski-Slominski1986}). Let ${\cal X}$ be a Polish space and let $X$ be ${\cal X}$-valued process with paths from the Skorokhod space $D({\cal X})$ (i.e. the space of RCLL functions $f: [0, \infty) \mapsto \cX$), adapted to the filtration $\bbf$. Process $Z$ with values from ${\cal P}({\cal X})$ and paths from $D({\cal P}({\cal X}))$ is called the prediction process of $(X, (\cF_t)_{t=0}^T)$ if $\forall t \geq 0 \, \forall A \in \cB(\cD(\cX))$
$$
Z(t, \omega)(A) = P(X \in A \, | \, \cF_t) \quad P{\rm - a.s.,}
$$ 
 i.e. $Z_t$ is a version of the regular conditional distribution of $X$ given $\cF_t$. It is proven (see, again \cite{Aldous1981, Jakubowski-Slominski1986})) that $Z$ is well-defined and 
$$
Z(\tau(\omega), \omega)(A) = P(X \in A \, | \, \cF_\tau ) \quad P{\rm - a.s.,}.
$$
for all $A \in \cB(\cD(\cX))$ finite stopping times $ \tau $.

The following theorem is a generalization of Theorem \ref{theo2}. We omit  assumption ${\bf H}$ by utilizing the prediction process techinque. 

\begin{theo}
    Let ${\bf M}$ and ${\bf \tilde  M}$ be two different models. Assume that the distributions $\cL_P(S, Y) = \cL_{\tilde P}(\tilde S, \tilde Y)$. Then $u(x, {\bf M}) = u(x, {\bf \tilde M})$. 
\end{theo}

\begin{proof}
    Let $B \in \cA(x, {\bf M})$. We construct its prediction process $Z^{B}(t, \omega)$. It is easily seen that $(B, Z^{B}) = (f(Y), g(Y))$ a.s. for some Borel functions $f: \cD^k_T \to \cD^d_T$ and $g: \cD^k_T \to  \cD(\cP(\cD^{d}_T))$ . We define $(\tilde B, \tilde Z^{B}) := (f(\tilde Y), g(\tilde Y))$. 
    
    Denote $\pi_u: \cD^{d}_T \mapsto \bbr^{d}$ the following projection mapping: $\pi_u(x):=x(u)$ and the mapping $\tilde \pi_u: \cP(\bbr^{d}) \mapsto \cP(\cD^{d}_T)$ as $\pi_u(\mu) := \mu \circ \pi_u^{-1}$.
    The prediction process has the following property(see \cite{Aldous1981,Jakubowski-Slominski1986}): there exists a set $ \Omega' \subseteq \Omega$ of probability one such that for all $\omega \in  \Omega'$ and $0 \leq u \leq t \leq T$
    $$
    \tilde \pi_u (Z^{B}(t, \omega)) = \delta_{B_u(\omega)}.
    $$
    As such, 
    $$
    \tilde \pi_u (g(y)(t)) = \delta_{f(y)(u)}, \quad \cL_P(Y)-{\rm a.e.}.
    $$
    Since any probability measure on a Polish space can be distinguished by a countable set of bounded continuous functions, we have the existence of a set $\tilde \Omega' \subseteq \tilde \Omega$ such that
    $$
    \tilde \pi_u (\tilde Z^{B}(t, \omega)) = \delta_{f(\tilde Y(\omega))(u)}.
    $$
     We aim to prove that $\tilde Z^{B} = Z^{\tilde B}$ a.s., i.e. that $\tilde Z^{B}$ is the prediction process of $(\tilde B, \tilde \bbf)$. To that end, we choose $t \in [0, T]$, a set $A \in \cB(\cD^{d}_T)$, a set $C \in \cB(\cD^k_T)$ and we verify that
    \beq
    \label{eq:ext}
    \tilde \E {\bf 1}_{\tilde B \in A} {\bf 1}_{\tilde Y^t \in C} = \tilde \E \tilde Z^{B}(t)(A) {\bf 1}_{Y^t \in C}.
    \eeq
    Immediately,
    $$
    \tilde \E {\bf 1}_{B \in A} {\bf 1}_{\tilde Y^t \in C} = \E {\bf 1}_{B \in A} {\bf 1}_{ Y^t \in C}.
    $$
    On the other hand, 
    $$
    \tilde \E \tilde Z^{B}(t)(A) {\bf 1}_{\tilde Y^t \in C} = \tilde \E \tilde g(\tilde Y)(t)(A) {\bf 1}_{\tilde Y^t \in C} =  \E  g( Y)(t)(A) {\bf 1}_{ Y^t \in C} = \E Z^{B}(t)(A) {\bf 1}_{ Y^t \in C}.
    $$
    By definition of the prediction process $Z^{B}$, 
    $$
     \E {\bf 1}_{B \in A} {\bf 1}_{ Y^t \in C} =  \E  Z^{B}(t)(A) {\bf 1}_{Y^t \in C},
    $$
    and \eqref{eq:ext} follows. 
    Also, since $\cL_{\tilde P}( Z^{\tilde B}) = \cL_{\tilde P}(\tilde Z^{B}) = \cL_P(Z^{B})$, we have for $t \in [0, T]$
    $$
    \tilde \pi_t (\tilde Z^{B} (t, \omega)) = \delta_{f(\tilde Y(\omega))(u))(u)}.
    $$
    By Lemma \ref{lemm:a1}, we obtain that $\tilde B$ is $\tilde \bbf$-adapted.  As in the proof of Theorem \ref{theo2}, we establish that $\tilde B$ is $K$-decreasing and satisfies the admissibility property. Finally, we obtain that $\tilde B \in \cA(x, {\tilde M})$. 
    Since $\cL_P(Y, S, B) = \cL_{\tilde P}(\tilde Y, \tilde S, \tilde B)$, we have
    $$
    \tilde \E \left( x + \frac{1}{\tilde S} \cdot \tilde B_T  \right) =  \E \left( x + \frac{1}{ S} \cdot  B_T  \right), 
    $$
    we have $u(x, {\bf M}) \leq u(x, {\bf \tilde M})$. By switching ${\bf M}$ and ${\bf \tilde M}$, we get the converse inequality.
\end{proof}

{\bf Acknowdgements.} I would like to express my sincere gratitude to my supervisor, Professor Yuri Kabanov, for his guidance and support throughout this research.

\section{Appendix}

\begin{lemm}
\label{lemm:var}
Let $X = (X_t)$ be an adapted  process with paths from $\cD^d_T$ and $K$ be a closed proper convex cone. Then the following properties are equivalent:

$(i)$ $X$ is $K$-decreasing;

$(ii)$ $X$ is of bounded variation with the Radon--Nikodym derivative $\dot X := d X / d {\rm Var}\, X$ evolving in $-K$ where 
${\rm Var}_t \, X:=\sum_{i\le d} {\rm Var}_t \, X^i$ is the total variation of $X$ on $[0, t]$. 
\end{lemm}
\begin{proof}
The implication $(ii) \Rightarrow (i)$ is obvious. Let us verify that $(i) \Rightarrow (ii)$. 
As $K$ is proper, $K^*$ has a non-empty interior and, therefore, there is a point $y\in K^*$ such that  $y+e^j \in K^*$ for all vectors $e^j$ of the canonical basis.  It follows that $K^*$ contains  $d$ linear independent vectors $a^j$. 
Define the scalar processes  $Z^j: = a^j X $. 
Since  $X_s - X_t \in K$ for $0 \leq s \leq t \leq T$, we have $Z^j_s - Z^j_t\ge 0$.  The coefficients $Y^j_t$, $j = 1, \dots, d$, of a linear combination
$$
X_t := \sum_{j=1}^d a^j Y_t^j 
$$ 
can be obtained as linear combinations of scalar products $Z_t$:
$$
Y_t = G^{-1} Z_t,
$$
where $G$ is the Gram matrix of the basis $(a^j)_{j=1}^d$. It implies that $Y^j$ are RCLL processes of bounded variation and so is $X$. By virtue of Lemma I.3.13 of \cite{JS}, the Radon--Nikodym derivatives $\dot X_t:=dX_t/d{\rm Var}_tX$ are well-defined. Since 
$$
a_i(X_s - X_t) = \int_{]s, t]} a \dot X_u d {\rm Var}_u X \le 0, \quad dP \otimes d {\rm Var}_T \, X-{\rm a.e.},
$$ 
$\dot X \in K$ up to a set of $dP \otimes d \, {\rm Var}_T \, X$-measure zero. By redefining $\dot X$ on this set, we arrive at the required statement. 
\end{proof}


\begin{lemm}
\label{lemm:a1}
    Let $\cA$ and $\cB$ be two Polish spaces and let $j: \cA \times \cB \to \bbr$ be a Borel mapping. Let $({\bf a}, {\bf b})$ be a $\cA \times \cB$-valued random variable. If the regular conditional distribution
    $$
    \nu ({\bf a}, C) := P(j({\bf a}, {\bf b}) \in C \, | \, \sigma\{ {\bf a} \}) = {\bf 1}_{j({\bf a}, {\bf b}) \in C}, \quad P-{\rm a.s.}
    $$
    for any $C \in \cB(\bbr)$ then $j({\bf a}, {\bf b}) = k({\bf a})$ $P$-a.s. for some Borel function $k$.
\end{lemm}
\begin{proof}
    Without loss of generality, we assume that $j$ is bounded. Then, 
    $$
    \E [j({\bf a}, {\bf b}) \, | \, \sigma\{ {\bf a} \}] = \int_\bbr x  \nu({\bf a}, dx) = j({\bf a}, {\bf b}), \quad P-{\rm a.s.}
    $$
    By definition, $\E [j({\bf a}, {\bf b}) \, | \, \sigma\{ {\bf a} \}]$ is $\sigma\{ {\bf a} \}$-measurable, and the assertion follows.
\end{proof}

\end{document}